\documentclass[reqno]{amsart}

\usepackage{amsmath, amsfonts, amssymb}

\newtheorem{theorem}{Theorem}
\newtheorem{lemma}{Lemma}
\newtheorem{proposition}{Proposition}
\newtheorem{corollary}{Corollary}

\title{Presentations of Lower-Triangular Subgroups of $\operatorname{Aut}(F_n)$}
\author{C. E. Kofinas}

\begin{document}

\begin{abstract}
Let $F_n$ be the free group of rank $n\geq2$ with basis
$x_1,\ldots,x_n$. For
$1\leq j<i\leq n$, let $d_{i,j}$ and $e_{i,j}$ be the automorphisms
of $F_n$ defined by $d_{i,j}(x_i)=x_ix_j$ and
$e_{i,j}(x_i)=x_jx_i$, respectively, and fixing the remaining free
generators. Write $D_n=\langle d_{i,j}\mid 1\leq j<i\leq n\rangle$ and
$A_n^+=\langle d_{i,j},e_{i,j}\mid 1\leq j<i\leq n\rangle$. 
We prove that $D_n$ admits a presentation on the generators
$d_{r+1,r}$, $1\leq r\leq n-1$, with three families of relations
given by commutators of weights two, three, and four, respectively.
We extend this to a presentation of $A_n^+$ on the generators
$d_{r+1,r}$ and $e_{r+1,r}$, $1\leq r\leq n-1$.
These generating sets have minimum cardinality.
Moreover, the presentation of $A_n^+$ may be chosen so that every
defining relator is a single commutator.

\medskip


\end{abstract}

\maketitle

\section{Introduction}

Let $F_n$ be the free group of rank $n\geq2$ with basis
$x_1,\ldots,x_n$. For $1\leq j<i\leq n$, let $d_{i,j}$ and $e_{i,j}$
be the automorphisms of $F_n$ defined by
\[
d_{i,j}(x_i)=x_ix_j,
\qquad
e_{i,j}(x_i)=x_jx_i,
\]
and fixing the other basis elements. Set
\[
D_n=\langle d_{i,j}\mid 1\leq j<i\leq n\rangle,
\qquad
D_n^{\ell}=\langle e_{i,j}\mid 1\leq j<i\leq n\rangle,
\qquad
A_n^+=\langle D_n,D_n^{\ell}\rangle.
\]
We call $d_{i,j}$ and $e_{i,j}$ adjacent if $i=j+1$, and
non-adjacent if $i\geq j+2$.

Erofeev and Roman'kov studied unitriangular automorphism groups of
relatively free groups and obtained, in particular, a normal form and
a presentation~\cite[Theorem~A and relations~(7)]{er}. In the
free-group case their group is $D_n^{\ell}$. The lower central series of
$D_n$ is determined in~\cite{kofinas-lcs}.
Satoh studied the lower-triangular IA-automorphism group $IA_n^+$ in~\cite{sat1} 
and later gave a normal form and a finite presentation of $A_n^+$ in
\cite[Lemma~3.1 and Theorem~3.2]{sat2}. Gersten gave a presentation
of the special automorphism group of $F_n$ in terms of Nielsen
automorphisms~\cite[Theorem~2.8]{gersten}. To our knowledge,
presentations of $D_n$ and $A_n^+$ on the adjacent Nielsen generators
have not previously been given. For presentations and further results
concerning related subgroups of $\operatorname{Aut}(F_n)$, see
\cite{mccool74,mccool86,cpvw,suw}.

The images of $D_n$, $D_n^{\ell}$ and $A_n^+$ on the abelianization
of $F_n$ coincide with the group $\Lambda_n$ of integral
lower-unitriangular matrices. Magnus obtained a presentation and a
normal form for $\Lambda_n$ in~\cite{mag}. 
Biss and Dasgupta~\cite{bd} later gave a presentation of the integral
upper-unitriangular group using the matrices with a $1$ in position
$(i,i+1)$ and zeros in all other off-diagonal positions,
$1\leq i\leq n-1$, as generators. This motivates our choice of
generators.

Every non-adjacent $d_{i,j}$ and $e_{i,j}$ can be expressed
recursively as an iterated commutator of adjacent Nielsen
automorphisms. Thus the presentations of Erofeev--Roman'kov and
Satoh can be transformed into presentations on these generators by
Tietze transformations. For Tietze transformations of group
presentations, see \cite{johnson}. 
Direct elimination, however, produces relations involving recursively
defined iterated commutators whose complexity increases with $n$.
Our first main result gives a presentation of $D_n$ with only three
families of defining relations. The generators are $d_{r+1,r}$,
$1\leq r\leq n-1$, and the three families consist of commutators 
of weights two, three, and four, respectively.

The presentation of $D_n$ is then extended to $A_n^+$. The relations
among the generators $d_{r+1,r}$ are precisely the defining relations
for $D_n$, while the corresponding relations among the generators
$e_{r+1,r}$ are obtained from those for $D_n$ by conjugation. We prove
that three additional families of mixed relations, involving at most
three of these generators, are sufficient. This gives a presentation
of $A_n^+$ on the generators $d_{r+1,r}$ and $e_{r+1,r}$,
$1\leq r\leq n-1$.

In both cases, the generating sets displayed above have minimum
cardinality. Moreover, the presentation of $A_n^+$ may be chosen so that every
defining relator is a single commutator.

\section{Preliminaries and notation}\label{sec-prelim}

\subsection{Notation}

Let $G$ be a group. For $x,y\in G$, write $x^y=y^{-1}xy$. The commutator of $x$ and $y$
is
\[
(x,y)=x^{-1}y^{-1}xy=x^{-1}x^y.
\]
Thus $(x,y)^{-1}=(y,x)$. For $r\geq3$, we use the left-normed
convention
\[
(x_1,\ldots,x_r)=((x_1,\ldots,x_{r-1}),x_r).
\]
Given $g_1,\ldots,g_t\in G$, we assign weight one to each symbol
$g_i^{\pm1}$. Recursively, if the commutator expressions $u$ and $v$
have weights $p$ and $q$, respectively, then $(u,v)$ has weight
$p+q$. Thus the weight is the number of occurrences of the symbols
$g_i^{\pm1}$, counted with repetitions, in the commutator expression.

For $x,y,z\in G$, we use the standard identities
\[
(xy,z)=(x,z)^y(y,z),
\qquad
(x,yz)=(x,z)(x,y)^z,
\]
\[
(x^{-1},y)=(y,x)^{x^{-1}},
\qquad
(x,y)^z=(x^z,y^z),
\]
and the Hall--Witt identity
\[
(x,y,z^x)(z,x,y^z)(y,z,x^y)=1.
\]

We denote the derived subgroup of $G$ by $G'$ and its abelianization
$G/G'$ by $G^{\mathrm{ab}}$. If $g_1,\ldots,g_t\in G$, then
$\langle g_1,\ldots,g_t\rangle$ denotes the subgroup generated by
these elements. If $A,B\leq G$, then $\langle A,B\rangle$ denotes
the subgroup generated by $A$ and $B$. We say that $A$ centralizes $B$ if
$(a,b)=1$ for every $a\in A$ and $b\in B$.

We compose automorphisms as functions, so
$(\alpha\beta)(x)=\alpha(\beta(x))$.

\subsection{The groups $D_m$ and $A_m^+$}

For $1\leq j<i\leq n$, let $d_{i,j},e_{i,j}\in
\operatorname{Aut}(F_n)$ be defined by
\[
d_{i,j}(x_i)=x_ix_j,
\qquad
e_{i,j}(x_i)=x_jx_i,
\]
and
\[
d_{i,j}(x_k)=e_{i,j}(x_k)=x_k
\qquad (k\neq i).
\]
For $m=2,\ldots,n$, set
\[
U_m=\langle d_{m,1},\ldots,d_{m,m-1}\rangle,
\qquad
H_m=\langle e_{m,1},\ldots,e_{m,m-1}\rangle,
\qquad
W_m=\langle U_m,H_m\rangle,
\]
and
\[
D_m=\langle U_2,\ldots,U_m\rangle,
\qquad
D_m^{\ell}=\langle H_2,\ldots,H_m\rangle,
\qquad
A_m^+=\langle D_m,D_m^{\ell}\rangle.
\]
All these groups are subgroups of $\operatorname{Aut}(F_n)$ and fix
$x_{m+1},\ldots,x_n$.

Let $\theta\in\operatorname{Aut}(F_n)$ be defined by
$\theta(x_r)=x_r^{-1}$ for $r=1,\ldots,n$. Then
$\theta d_{i,j}\theta^{-1}=e_{i,j}$, so
$\theta D_m\theta^{-1}=D_m^{\ell}$ and
$\theta U_m\theta^{-1}=H_m$.
The group $D_m^{\ell}$ is the unitriangular automorphism group studied
by Erofeev and Roman'kov \cite{er}. In Satoh's notation,
$d_{i,j}=E_{ij}$ and $e_{i,j}=E_{i^{-1}j}^{-1}$, and $A_m^+$ is the
lower-triangular automorphism group considered in
\cite[Section~3]{sat2}.

We shall use the following known structure results. The freeness of
$H_m$ and the decomposition of $D_m^{\ell}$ follow from
\cite[Theorem~A]{er}. The assertions for $U_m$ and $D_m$ 
follow by conjugation with
$\theta$. The direct-product decomposition of
$W_m$ and the decomposition of $A_m^+$ follow from Satoh's normal
form~\cite[Lemma~3.1]{sat2}.

\begin{proposition}\label{triangular-row-decompositions}
For every $2\leq m\leq n$, the groups $U_m$ and $H_m$ are free of
rank $m-1$, freely generated by
$d_{m,1},\ldots,d_{m,m-1}$ and
$e_{m,1},\ldots,e_{m,m-1}$, respectively, and
\[
W_m=U_m\times H_m.
\]
For $3\leq m\leq n$,
\[
D_m=U_m\rtimes D_{m-1},
\qquad
D_m^{\ell}=H_m\rtimes D_{m-1}^{\ell},
\qquad
A_m^+=W_m\rtimes A_{m-1}^+.
\]
\end{proposition}

Erofeev and Roman'kov give the commutator relations for the
$e_{i,j}$ in~\cite[formulas~(7)]{er}. 
After rewriting those formulas in our commutator convention, we obtain
the relations for the $e_{i,j}$ below. 
The relations for the $d_{i,j}$ follow by conjugation with $\theta$.

\begin{lemma}\label{nielsen-commutator-relations}
Let $1\leq l<k<i\leq n$ and $1\leq j<i$. Then
\[
\begin{aligned}
(d_{i,j},d_{k,l})
&=
\begin{cases}
d_{i,l}^{-1},&j=k,\\
1,&j\neq k,
\end{cases}
&
(e_{i,j},e_{k,l})
&=
\begin{cases}
e_{i,l}^{-1},&j=k,\\
1,&j\neq k.
\end{cases}
\end{aligned}
\]
\end{lemma}
\begin{lemma}\label{adjacent-generators}
For $3\leq i\leq n$ and $1\leq j\leq i-2$,
\[
\begin{split}
d_{i,j}
&=(d_{j+1,j},d_{j+2,j+1},\ldots,d_{i,i-1}),\\
e_{i,j}
&=(e_{j+1,j},e_{j+2,j+1},\ldots,e_{i,i-1}).
\end{split}
\]
Moreover,
\[
\begin{split}
D_n
&=\langle d_{r+1,r}\mid 1\leq r\leq n-1\rangle,\\
D_n^{\ell}
&=\langle e_{r+1,r}\mid 1\leq r\leq n-1\rangle,\\
A_n^+
&=\langle d_{r+1,r},e_{r+1,r}\mid 1\leq r\leq n-1\rangle.
\end{split}
\]
\end{lemma}

\begin{proof}
For $3\leq i\leq n$ and $1\leq j\leq i-2$,
Lemma~\ref{nielsen-commutator-relations} gives
$d_{i,j}=(d_{i-1,j},d_{i,i-1})$.
Iterating this identity gives
\[
d_{i,j}
=(d_{j+1,j},d_{j+2,j+1},\ldots,d_{i,i-1}).
\]
The formula for $e_{i,j}$ follows by conjugation with $\theta$.

Thus
\[
\begin{aligned}
D_n
&=\langle d_{i,j}\mid 1\leq j<i\leq n\rangle =\langle d_{2,1},d_{3,2},\ldots,d_{n,n-1}\rangle,
\\
D_n^{\ell}
&=\langle e_{i,j}\mid 1\leq j<i\leq n\rangle =\langle e_{2,1},e_{3,2},\ldots,e_{n,n-1}\rangle.
\end{aligned}
\]
Since $A_n^+=\langle D_n,D_n^{\ell}\rangle$, the adjacent right and
left Nielsen automorphisms generate $A_n^+$.
\end{proof}

\section{A presentation of $D_n$}\label{sec-Dn-presentation}

\subsection{The defining relations of $D_n$}

\begin{theorem}\label{main-adjacent-presentation}
For every $n\geq2$, the group $D_n$ has a presentation with generators
$d_{2,1},d_{3,2},\ldots,d_{n,n-1}$
and the following defining relations:
\begin{enumerate}
\item[\textup{(R1)}]
$(d_{r+1,r},d_{s+1,s})=1$ for
$1\leq r<s\leq n-1$ and $s-r\geq2$;

\item[\textup{(R2)}]
$(d_{r+1,r},(d_{r+1,r},d_{r+2,r+1}))=1$ for
$1\leq r\leq n-2$;

\item[\textup{(R3)}]
$(d_{r+1,r},d_{r+2,r+1},
(d_{r+2,r+1},d_{r+3,r+2}))=1$ for
$1\leq r\leq n-3$.
\end{enumerate}
\end{theorem}

The commutators in~\textup{(R1)}, \textup{(R2)}, and
\textup{(R3)} have weights two, three, and four, respectively.

\begin{lemma}\label{Dn-relations-valid}
Let $n\geq2$. Then relations~\textup{(R1)}--\textup{(R3)} in
Theorem~\ref{main-adjacent-presentation} hold in $D_n$.
\end{lemma}

\begin{proof}
If $1\leq r<s\leq n-1$ and $s-r\geq2$, then
Lemma~\ref{nielsen-commutator-relations} gives
$(d_{s+1,s},d_{r+1,r})=1$. Hence
$(d_{r+1,r},d_{s+1,s})=1$. This proves~\textup{(R1)}.

For $1\leq r\leq n-2$, the same lemma gives
$(d_{r+1,r},d_{r+2,r+1})=d_{r+2,r}$ and
$(d_{r+2,r},d_{r+1,r})=1$. Therefore
\[
(d_{r+1,r},(d_{r+1,r},d_{r+2,r+1}))
=(d_{r+1,r},d_{r+2,r})=1.
\]
This proves~\textup{(R2)}.

Finally, for $1\leq r\leq n-3$, the same lemma gives
$(d_{r+2,r+1},d_{r+3,r+2})=d_{r+3,r+1}$ and
$(d_{r+3,r+1},d_{r+2,r})=1$. It follows that
\[
(d_{r+1,r},d_{r+2,r+1},
(d_{r+2,r+1},d_{r+3,r+2}))
=(d_{r+2,r},d_{r+3,r+1})=1.
\]
This proves~\textup{(R3)}.
\end{proof}

\subsection{Consequences of the Hall--Witt identity}

\begin{lemma}\label{hallwitt-rebracketing}
Let $G$ be a group and let $x,y,z\in G$. Suppose that
\[
(x,z)=1,\qquad (x,(x,y))=1,\qquad (x,y,(y,z))=1.
\]
Then
\[
(x,y,z)=(x,(y,z)).
\]
\end{lemma}

\begin{proof}
The Hall--Witt identity gives
\[
(x,y,z^x)(z,x,y^z)(y,z,x^y)=1.
\]
Since $(x,z)=1$, we have $z^x=z$ and $(z,x)=1$, so this reduces to
\[
(x,y,z)(y,z,x^y)=1.
\]
Write $A=(y,z)$ and $B=(x,y)$. Since $x^y=xB$,
\[
(y,z,x^y)=(A,xB)=(A,B)(A,x)^B.
\]
The third hypothesis shows that $A$ and $B$ commute, while the second
shows that $B$ commutes with $x$. Hence $B$ commutes with $(A,x)$,
and therefore
\[
(y,z,x^y)=(A,x)=(y,z,x).
\]
Thus $(x,y,z)(y,z,x)=1$, and hence $(x,y,z)=(x,(y,z))$.
\end{proof}
\begin{lemma}\label{local-four-criterion}
Let $G$ be a group and let $x,y,z\in G$. Suppose that
\[
(x,(x,y))=1,\qquad (y,(y,z))=1,\qquad (x,y,(y,z))=1.
\]
Then
\[
(x,(y,z),y)=1.
\]
\end{lemma}

\begin{proof}
Write $B=(x,y)$, $C=(y,z)$, and $A=(x,C)$. The hypotheses give
\[
(x,B)=(y,C)=(B,C)=1.
\]
Apply the Hall--Witt identity in the form
\[
(x,C,y^x)(y,x,C^y)(C,y,x^C)=1.
\]
Since $(y,C)=1$, the last factor is trivial and $C^y=C$. Moreover,
$(y,x)=B^{-1}$ and $(B,C)=1$, so the middle factor is also trivial.
Hence $(x,C,y^x)=1$.

Now $y^x=yB^{-1}$, and therefore
\[
1=(A,yB^{-1})=(A,B^{-1})(A,y)^{B^{-1}}.
\]
Since $B$ commutes with both $x$ and $C$, it commutes with
$A=(x,C)$. Thus $(A,B^{-1})=1$, and consequently $(A,y)=1$, as
required.
\end{proof}

\begin{lemma}\label{tail-action-Dn}
Let $G$ be a group, let $m\geq3$, and let
$a_1,\ldots,a_{m-1}\in G$. Suppose that the following conditions
hold:
\begin{enumerate}
\item[\textup{(A1)}]
$(a_r,a_s)=1$ for
$1\leq r<s\leq m-1$ and $s-r\geq2$;

\item[\textup{(A2)}]
$(a_r,(a_r,a_{r+1}))=1$ for
$1\leq r\leq m-2$;

\item[\textup{(A3)}]
$(a_r,a_{r+1},(a_{r+1},a_{r+2}))=1$ for
$1\leq r\leq m-3$.
\end{enumerate}
Write $u_{m-1}=a_{m-1}$ and, for $j=m-2,\ldots,1$, define recursively
\[
u_j=(a_j,u_{j+1}).
\]
Then:
\begin{enumerate}
\item[\textup{(C1)}]
For $1\leq r\leq m-2$,
$(u_{r+1},a_r)=u_r^{-1}$ and $(u_j,a_r)=1$ whenever
$1\leq j\leq m-1$ and $j\neq r+1$;

\item[\textup{(C2)}]
$(a_j,a_{j+1},u_{j+1})=1$ for $1\leq j\leq m-3$;

\item[\textup{(C3)}]
$u_j=(a_j,a_{j+1},u_{j+2})$ for $1\leq j\leq m-3$.
\end{enumerate}
\end{lemma}

\begin{proof}
We first prove~\textup{(C2)} and~\textup{(C3)}, together with the
commutativity relations
\begin{equation}\label{eq-tail-commuting}
\begin{gathered}
(u_j,a_j)=1\quad (1\leq j\leq m-2),
\\
(u_j,a_{j+1})=1\quad (1\leq j\leq m-3).
\end{gathered}
\end{equation}
Since $u_{m-2}=(a_{m-2},a_{m-1})$, condition~\textup{(A2)} shows
that $u_{m-2}$ commutes with $a_{m-2}$. If $m=3$, this gives
$(u_1,a_1)=1$, while $(u_2,a_1)=u_1^{-1}$ follows from
$u_1=(a_1,u_2)$. This proves the case $m=3$.

Assume that $m\geq4$. For $j=m-3,\ldots,1$, we prove the identities
in~\textup{(C2)} and~\textup{(C3)} and the two commutativity
relations in~\eqref{eq-tail-commuting} by descending induction.
For $j=m-3$, we have
$u_{j+1}=(a_{j+1},a_{j+2})$, so condition~\textup{(A3)} gives
$(a_j,a_{j+1},u_{j+1})=1$. Lemma~\ref{hallwitt-rebracketing}
applies with $x=a_j$, $y=a_{j+1}$ and $z=a_{j+2}$. Its first
hypothesis follows from condition~\textup{(A1)} by taking $r=j$
and $s=j+2$, its second hypothesis follows from
condition~\textup{(A2)} with $r=j$, and its third hypothesis follows
from condition~\textup{(A3)} with $r=j$. Hence
\[
u_j=(a_j,a_{j+1},u_{j+2}).
\]
By condition~\textup{(A2)}, $a_j$ commutes with
$(a_j,a_{j+1})$, and by condition~\textup{(A1)}, it commutes with
$u_{j+2}=a_{j+2}$. Therefore $(u_j,a_j)=1$.

Lemma~\ref{local-four-criterion} also applies with $x=a_j$,
$y=a_{j+1}$ and $z=a_{j+2}$. Its first and second hypotheses follow
from condition~\textup{(A2)} with $r=j$ and $r=j+1$, respectively,
while its third hypothesis follows from condition~\textup{(A3)}
with $r=j$. Therefore $(u_j,a_{j+1})=1$.

Now let $1\leq j\leq m-4$ and assume that the identities
in~\textup{(C2)} and~\textup{(C3)} and the two commutativity
relations hold with $j+1$ in place of $j$. By the induction
hypothesis,
$u_{j+1}=(a_{j+1},a_{j+2},u_{j+3})$. Condition~\textup{(A3)} gives
\[
(a_j,a_{j+1},(a_{j+1},a_{j+2}))=1.
\]
Moreover, by condition~\textup{(A1)},
$\langle a_j,a_{j+1}\rangle$ centralizes
$\langle a_{j+3},\ldots,a_{m-1}\rangle$, which contains
$u_{j+3}$. Thus $(a_j,a_{j+1})$ commutes with both
$(a_{j+1},a_{j+2})$ and $u_{j+3}$, and hence with $u_{j+1}$.
Therefore
\[
(a_j,a_{j+1},u_{j+1})=1.
\]

Apply Lemma~\ref{hallwitt-rebracketing} with
$x=a_j$, $y=a_{j+1}$ and $z=u_{j+2}$. The recursive definition gives
$u_{j+2}\in\langle a_{j+2},\ldots,a_{m-1}\rangle$.
Condition~\textup{(A1)} shows that $a_j$ centralizes this subgroup,
and hence $(a_j,u_{j+2})=1$. The second hypothesis follows from
condition~\textup{(A2)} with $r=j$, and the third hypothesis is the
identity just proved. Thus
\[
u_j=(a_j,u_{j+1})
   =(a_j,a_{j+1},u_{j+2}).
\]
Condition~\textup{(A2)} shows that $a_j$ commutes with
$(a_j,a_{j+1})$, while condition~\textup{(A1)} shows that it
commutes with $u_{j+2}$. Hence $(u_j,a_j)=1$.

Finally, apply Lemma~\ref{local-four-criterion} with
$x=a_j$, $y=a_{j+1}$ and $z=u_{j+2}$. Its first hypothesis follows
from condition~\textup{(A2)} with $r=j$. Its second hypothesis is
$(a_{j+1},u_{j+1})=1$, which follows
from~\eqref{eq-tail-commuting} with index $j+1$, and its third
hypothesis is $(a_j,a_{j+1},u_{j+1})=1$, proved above. It follows
that $(u_j,a_{j+1})=1$. This completes the descending induction.

We now prove~\textup{(C1)}. If $j=r+1$, then
$u_r=(a_r,u_{r+1})$, and therefore
$(u_{r+1},a_r)=u_r^{-1}$. If $j\geq r+2$, then
$u_j\in\langle a_j,\ldots,a_{m-1}\rangle$, and
condition~\textup{(A1)} shows that $a_r$ centralizes this subgroup.
Thus $(u_j,a_r)=1$.

The case $j=r$ follows from the first identity
in~\eqref{eq-tail-commuting}. If $r\geq2$, the case $j=r-1$ follows
from the second identity. Finally, suppose that
$1\leq j\leq r-2$. We proceed by descending induction, starting with
$(u_{r-1},a_r)=1$. If $(u_{k+1},a_r)=1$ for some
$1\leq k\leq r-2$, then $(a_k,a_r)=1$ by
condition~\textup{(A1)}. Consequently, $a_r$ commutes with both
entries in $u_k=(a_k,u_{k+1})$, and hence $(u_k,a_r)=1$.
This proves~\textup{(C1)}.
\end{proof}

\subsection{The presentation theorem for $D_n$}

\begin{proof}[Proof of Theorem~\ref{main-adjacent-presentation}]
By Lemma~\ref{adjacent-generators}, the elements
$d_{2,1},d_{3,2},\ldots,d_{n,n-1}$ generate $D_n$, and by
Lemma~\ref{Dn-relations-valid}, they satisfy
relations~\textup{(R1)}--\textup{(R3)}. It remains to prove that
these relations are sufficient. We proceed by induction on $n$.

For $n=2$, the presentation has one generator and no defining
relations. Since $D_2=U_2$ is infinite cyclic, generated by
$d_{2,1}$, the result follows.

Assume that $n\geq3$ and that the theorem holds for $D_{n-1}$. For
$1\leq r\leq n-1$, write $a_r=d_{r+1,r}$, and for
$1\leq j\leq n-1$, write $u_j=d_{n,j}$. By
Proposition~\ref{triangular-row-decompositions},
\[
D_n=U_n\rtimes D_{n-1},
\]
where $U_n$ is free on $u_1,\ldots,u_{n-1}$.
Lemma~\ref{nielsen-commutator-relations} gives, for
$1\leq r\leq n-2$ and $1\leq j\leq n-1$,
\begin{equation}\label{eq-Dn-action}
(u_j,a_r)=
\begin{cases}
u_r^{-1},&j=r+1,\\
1,&j\neq r+1.
\end{cases}
\end{equation}
The standard presentation of a semidirect product therefore gives a
presentation of $D_n$ by adjoining to the presentation of
$D_{n-1}$ the free generators $u_1,\ldots,u_{n-1}$ and the
relations~\eqref{eq-Dn-action}.

For each $1\leq r\leq n-2$, the nontrivial relation
in~\eqref{eq-Dn-action} is equivalent to
\[
u_r=(a_r,u_{r+1}).
\]
We use these relations successively to eliminate
$u_{n-2},u_{n-3},\ldots,u_1$,
and rename $u_{n-1}$ as $a_{n-1}$. After these eliminations, we use
$u_j$ to denote the recursive words
\[
u_{n-1}=a_{n-1},
\qquad
u_j=(a_j,u_{j+1})
\quad (1\leq j\leq n-2).
\]
The relations remaining from~\eqref{eq-Dn-action} are
\[
(u_j,a_r)=1
\qquad
(1\leq r\leq n-2,\ 1\leq j\leq n-1,\ j\neq r+1).
\]

The preceding eliminations are Tietze transformations, so the current
presentation still defines $D_n$. Consequently, every relation among
its generators that holds in $D_n$ is a consequence of the current
relators and may be added by a Tietze transformation. By
Lemma~\ref{Dn-relations-valid}, we may therefore add all the relations
in~\textup{(R1)}--\textup{(R3)} that involve $a_{n-1}$. All other
relations in~\textup{(R1)}--\textup{(R3)} already occur in the
presentation of $D_{n-1}$.

The group defined by the current presentation now satisfies
relations~\textup{(R1)}--\textup{(R3)}. With $m=n$, these are
precisely conditions~\textup{(A1)}--\textup{(A3)} of
Lemma~\ref{tail-action-Dn}. Conclusion~\textup{(C1)} of that lemma
shows that all the remaining relations from~\eqref{eq-Dn-action}
follow from~\textup{(R1)}--\textup{(R3)}. They may therefore be
deleted by Tietze transformations. What remains is exactly the
presentation in the statement.
\end{proof}

\begin{corollary}\label{minimal-generating-set}
For every $n\geq2$,
$D_n^{\mathrm{ab}}\cong\mathbb Z^{n-1}$. Every generating set of
$D_n$ contains at least $n-1$ elements, and
$\{d_{2,1},d_{3,2},\ldots,d_{n,n-1}\}$
is a generating set of minimum cardinality.
\end{corollary}

\begin{proof}
All the defining relations in
Theorem~\ref{main-adjacent-presentation} are commutator relations and
therefore impose no relations in the abelianization. Hence the images
of the elements $d_{2,1},d_{3,2},\ldots,d_{n,n-1}$
form a basis of $D_n^{\mathrm{ab}}$, and therefore
\[
D_n^{\mathrm{ab}}\cong\mathbb Z^{n-1}.
\]

The image of every generating set of $D_n$ generates
$D_n^{\mathrm{ab}}$. Since $\mathbb Z^{n-1}$ cannot be generated by
fewer than $n-1$ elements, every generating set of $D_n$ contains at
least $n-1$ elements. The set
$\{d_{2,1},d_{3,2},\ldots,d_{n,n-1}\}$
has $n-1$ elements and generates $D_n$, and hence has minimum
cardinality.
\end{proof}

\section{A presentation of the lower-triangular group $A_n^+$}
\label{sec-Aplus-presentation}

We now pass from $D_n$ to $A_n^+$. Recall that $\theta$ is the
automorphism of $F_n$ defined by
$\theta(x_i)=x_i^{-1}$ for $1\leq i\leq n$.
Conjugation by $\theta$ sends each $d_{i,j}$ to $e_{i,j}$ and maps
$D_n$ isomorphically onto $D_n^{\ell}$. For
$1\leq r\leq n-1$, write $a_r=d_{r+1,r}$ and
$b_r=e_{r+1,r}$. Thus the subgroup generated by the $a_r$ is $D_n$,
while the subgroup generated by the $b_r$ is $D_n^{\ell}$.

\subsection{The defining relations of $A_n^+$}

The relations involving only the $a_r$ are the defining relations
from Theorem~\ref{main-adjacent-presentation}, and the same relations
hold with the $b_r$ in place of the $a_r$. For $n\geq3$,
Proposition~\ref{triangular-row-decompositions} gives
\[
A_n^+=W_n\rtimes A_{n-1}^+,
\qquad
W_n=U_n\times H_n.
\]
Satoh's finite presentation uses all $n(n-1)$ elementary right and
left generators~\cite[Theorem~3.2]{sat2}. The following theorem gives
a presentation on the $2(n-1)$ adjacent generators, with three
families of mixed relations in addition to the two copies of the
relations for $D_n$.

\begin{theorem}\label{main-adjacent-presentation-Aplus}
For every $n\geq2$, write
\[
a_r=d_{r+1,r},\qquad b_r=e_{r+1,r}
\qquad (1\leq r\leq n-1).
\]
Then $A_n^+$ has a presentation with generators
\[
a_1,\ldots,a_{n-1},b_1,\ldots,b_{n-1}
\]
and defining relations
\[
\begin{aligned}
\textup{(S1)}\quad
&(a_r,a_s)=1
&& (1\leq r<s\leq n-1,\ s-r\geq2),
\\
&(a_r,(a_r,a_{r+1}))=1
&& (1\leq r\leq n-2),
\\
&(a_r,a_{r+1},(a_{r+1},a_{r+2}))=1
&& (1\leq r\leq n-3),
\\
&(b_r,b_s)=1
&& (1\leq r<s\leq n-1,\ s-r\geq2),
\\
&(b_r,(b_r,b_{r+1}))=1
&& (1\leq r\leq n-2),
\\
&(b_r,b_{r+1},(b_{r+1},b_{r+2}))=1
&& (1\leq r\leq n-3),
\\[1mm]
\textup{(S2)}\quad
&(a_r,b_s)=1
&& (1\leq r,s\leq n-1,\ |r-s|\neq1),
\\[1mm]
\textup{(S3)}\quad
&(a_r,a_{r+1})=(a_{r+1}^{-1},b_r)
&& (1\leq r\leq n-2),
\\
&(b_r,b_{r+1})=(b_{r+1}^{-1},a_r)
&& (1\leq r\leq n-2),
\\[1mm]
\textup{(S4)}\quad
&(a_r,a_{r+1},b_{r+1})=1
&& (1\leq r\leq n-2),
\\
&(b_r,b_{r+1},a_{r+1})=1
&& (1\leq r\leq n-2).
\end{aligned}
\]
\end{theorem}

The relations in~\textup{(S1)} are two copies of
relations~\textup{(R1)}--\textup{(R3)} for $D_n$. The mixed
relations~\textup{(S2)}--\textup{(S4)} involve at most three
distinct adjacent generators.

The following formulas will be used both to verify the mixed
relations and to construct the semidirect-product presentation.

\begin{lemma}\label{row-action-Aplus}
Let $3\leq i\leq n$. For $1\leq j\leq i-1$, write
\[
u_j=d_{i,j},\qquad v_j=e_{i,j}.
\]
For $1\leq r\leq i-2$, write
\[
a_r=d_{r+1,r},\qquad b_r=e_{r+1,r}.
\]
Then, for $1\leq r\leq i-2$ and $1\leq j\leq i-1$, the following
hold:
\begin{enumerate}
\item
\[
\begin{aligned}
(u_j,a_r)&=
\begin{cases}
u_r^{-1},&j=r+1,\\
1,&j\neq r+1,
\end{cases}
&
(v_j,b_r)&=
\begin{cases}
v_r^{-1},&j=r+1,\\
1,&j\neq r+1.
\end{cases}
\end{aligned}
\]

\item
\[
\begin{aligned}
u_j^{b_r}&=
\begin{cases}
u_r^{-1}u_{r+1},&j=r+1,\\
u_j,&j\neq r+1,
\end{cases}
&
v_j^{a_r}&=
\begin{cases}
v_r^{-1}v_{r+1},&j=r+1,\\
v_j,&j\neq r+1.
\end{cases}
\end{aligned}
\]
The two nontrivial cases can equivalently be written
as
\[
(d_{i,r+1}^{-1},e_{r+1,r})=d_{i,r},
\qquad
(e_{i,r+1}^{-1},d_{r+1,r})=e_{i,r}.
\]
\end{enumerate}
\end{lemma}

\begin{proof}
Since $1\leq r\leq i-2$, we have $1\leq r<r+1<i$. Applying
Lemma~\ref{nielsen-commutator-relations} with $k=r+1$ and $l=r$
therefore gives the formula in part~(1) involving $u_j$
and $a_r$. Conjugation by $\theta$ sends $u_j$ to $v_j$ and $a_r$
to $b_r$. Therefore the formula involving $v_j$ and $b_r$ follows.
This proves part~(1).

We now prove part~(2). Suppose first that $j\neq r+1$.
The automorphism $b_r$ fixes $x_i$ and $x_j$, while $u_j$ fixes
$x_r$ and $x_{r+1}$. It follows directly from their actions on the
free basis that $u_j$ and $b_r$ commute. Similarly, $a_r$ fixes
$x_i$ and $x_j$, while $v_j$ fixes $x_r$ and $x_{r+1}$, so $v_j$
and $a_r$ commute. Therefore, for $j\neq r+1$, $u_j^{b_r}=u_j$ and
$v_j^{a_r}=v_j$.

It remains to consider $j=r+1$. Direct calculation gives
\[
\begin{aligned}
(d_{i,r+1}^{-1},e_{r+1,r})(x_i)
&=d_{i,r+1}e_{r+1,r}^{-1}
  (x_ix_{r+1}^{-1})\\
&=d_{i,r+1}(x_ix_{r+1}^{-1}x_r)
 =x_ix_r,\\
(d_{i,r+1}^{-1},e_{r+1,r})(x_{r+1})
&=d_{i,r+1}e_{r+1,r}^{-1}(x_rx_{r+1})\\
&=d_{i,r+1}(x_{r+1})
 =x_{r+1}.
\end{aligned}
\]
All other basis elements are fixed, so
$(d_{i,r+1}^{-1},e_{r+1,r})=d_{i,r}$.

For the second mixed identity,
\[
\begin{aligned}
(e_{i,r+1}^{-1},d_{r+1,r})(x_i)
&=e_{i,r+1}d_{r+1,r}^{-1}
  (x_{r+1}^{-1}x_i)\\
&=e_{i,r+1}(x_rx_{r+1}^{-1}x_i)
 =x_rx_i,\\
(e_{i,r+1}^{-1},d_{r+1,r})(x_{r+1})
&=e_{i,r+1}d_{r+1,r}^{-1}
  (x_{r+1}x_r)\\
&=e_{i,r+1}(x_{r+1})
 =x_{r+1}.
\end{aligned}
\]
Again, all other basis elements are fixed, and hence
$(e_{i,r+1}^{-1},d_{r+1,r})=e_{i,r}$.

Finally, since $(x^{-1},y)=x(x^y)^{-1}$, the first identity is
equivalent to
\[
u_r=u_{r+1}(u_{r+1}^{b_r})^{-1},
\]
and hence to $u_{r+1}^{b_r}=u_r^{-1}u_{r+1}$. Similarly, the second
identity is equivalent to
\[
v_r=v_{r+1}(v_{r+1}^{a_r})^{-1},
\]
and hence to $v_{r+1}^{a_r}=v_r^{-1}v_{r+1}$. Together with the cases
$j\neq r+1$, this proves part~(2).
\end{proof}

\begin{lemma}\label{Aplus-relations-valid}
Let $n\geq2$. For $1\leq r\leq n-1$, write
\[
a_r=d_{r+1,r},\qquad b_r=e_{r+1,r}.
\]
Then relations~\textup{(S1)}--\textup{(S4)} in
Theorem~\ref{main-adjacent-presentation-Aplus} hold in $A_n^+$.
\end{lemma}

\begin{proof}
The three families of relations in~\textup{(S1)} involving only the
$a_r$ follow from
Theorem~\ref{main-adjacent-presentation}. Conjugation by $\theta$
sends $a_r$ to $b_r$, so the same three relations hold with the
$b_r$ in place of the $a_r$. Thus all the relations
in~\textup{(S1)} hold.

We next prove~\textup{(S2)}. For $1\leq r\leq n-1$, the products
$a_rb_r$ and $b_ra_r$ fix every basis element except possibly
$x_{r+1}$, and
\[
a_rb_r(x_{r+1})
 =x_rx_{r+1}x_r
 =b_ra_r(x_{r+1}).
\]
Hence $(a_r,b_r)=1$.

Suppose now that $1\leq r,s\leq n-1$ and $|r-s|\geq2$.
Both products $a_rb_s$ and $b_sa_r$
fix every basis element other than $x_{r+1}$ and $x_{s+1}$, and
\[
a_rb_s(x_{r+1})
 =x_{r+1}x_r
 =b_sa_r(x_{r+1}),
\]
while
\[
a_rb_s(x_{s+1})
 =x_sx_{s+1}
 =b_sa_r(x_{s+1}).
\]
Therefore $(a_r,b_s)=1$. This proves~\textup{(S2)}, since the
condition $|r-s|\neq1$ consists precisely of the cases $r=s$ and
$|r-s|\geq2$.

To prove~\textup{(S3)}, let $1\leq r\leq n-2$. Then
Lemma~\ref{nielsen-commutator-relations} gives
\[
(a_r,a_{r+1})=d_{r+2,r},
\]
and conjugation by $\theta$ gives
\[
(b_r,b_{r+1})=e_{r+2,r}.
\]
Take $i=r+2$ in Lemma~\ref{row-action-Aplus}. In the notation of
that lemma,
\[
u_{r+1}=a_{r+1},\qquad u_r=d_{r+2,r},
\]
and
\[
v_{r+1}=b_{r+1},\qquad v_r=e_{r+2,r}.
\]
The two equivalent identities in part~(2) of that lemma
therefore give
\[
(a_{r+1}^{-1},b_r)=d_{r+2,r},
\qquad
(b_{r+1}^{-1},a_r)=e_{r+2,r}.
\]
Consequently,
\[
(a_r,a_{r+1})=(a_{r+1}^{-1},b_r)
\]
and
\[
(b_r,b_{r+1})=(b_{r+1}^{-1},a_r).
\]
This proves~\textup{(S3)}.

It remains to prove~\textup{(S4)}. For $1\leq r\leq n-2$,
\[
(a_r,a_{r+1})=d_{r+2,r},
\qquad
(b_r,b_{r+1})=e_{r+2,r}.
\]
The automorphisms $d_{r+2,r}$ and
$b_{r+1}=e_{r+2,r+1}$ fix every basis element except possibly
$x_{r+2}$, and
\[
d_{r+2,r}b_{r+1}(x_{r+2})
 =x_{r+1}x_{r+2}x_r
 =b_{r+1}d_{r+2,r}(x_{r+2}).
\]
Thus
\[
(a_r,a_{r+1},b_{r+1})=1.
\]

Similarly, $e_{r+2,r}$ and
$a_{r+1}=d_{r+2,r+1}$ fix every basis element except possibly
$x_{r+2}$, and
\[
e_{r+2,r}a_{r+1}(x_{r+2})
 =x_rx_{r+2}x_{r+1}
 =a_{r+1}e_{r+2,r}(x_{r+2}).
\]
Hence
\[
(b_r,b_{r+1},a_{r+1})=1.
\]
This proves~\textup{(S4)}.
\end{proof}

\subsection{Conjugation and commutation of recursively defined elements}

\begin{lemma}\label{mixed-action-Aplus}
Let $G$ be a group, let $n\geq3$, and let
\[
a_1,\ldots,a_{n-1},b_1,\ldots,b_{n-1}\in G.
\]
Suppose that the following relations hold:
\[
\begin{aligned}
\textup{(B1)}\quad
&(a_r,a_s)=1
&& (1\leq r<s\leq n-1,\ s-r\geq2),
\\
&(a_r,(a_r,a_{r+1}))=1
&& (1\leq r\leq n-2),
\\
&(a_r,a_{r+1},(a_{r+1},a_{r+2}))=1
&& (1\leq r\leq n-3),
\\
&(b_r,b_s)=1
&& (1\leq r<s\leq n-1,\ s-r\geq2),
\\
&(b_r,(b_r,b_{r+1}))=1
&& (1\leq r\leq n-2),
\\
&(b_r,b_{r+1},(b_{r+1},b_{r+2}))=1
&& (1\leq r\leq n-3),
\\[1mm]
\textup{(B2)}\quad
&(a_r,b_s)=1
&& (1\leq r,s\leq n-1,\ |r-s|\neq1),
\\[1mm]
\textup{(B3)}\quad
&(a_r,a_{r+1})=(a_{r+1}^{-1},b_r)
&& (1\leq r\leq n-2),
\\
&(b_r,b_{r+1})=(b_{r+1}^{-1},a_r)
&& (1\leq r\leq n-2),
\\[1mm]
\textup{(B4)}\quad
&(a_r,a_{r+1},b_{r+1})=1
&& (1\leq r\leq n-2),
\\
&(b_r,b_{r+1},a_{r+1})=1
&& (1\leq r\leq n-2).
\end{aligned}
\]
Set
\[
u_{n-1}=a_{n-1},\qquad v_{n-1}=b_{n-1},
\]
and, for $j=n-2,\ldots,1$, define recursively
\[
u_j=(a_j,u_{j+1}),\qquad
v_j=(b_j,v_{j+1}).
\]
Then, for $1\leq r\leq n-2$ and $1\leq j\leq n-1$, the following
hold:
\begin{enumerate}
\item
\[
u_j^{b_r}=
\begin{cases}
u_r^{-1}u_{r+1},&j=r+1,\\
u_j,&j\neq r+1.
\end{cases}
\]

\item
\[
v_j^{a_r}=
\begin{cases}
v_r^{-1}v_{r+1},&j=r+1,\\
v_j,&j\neq r+1.
\end{cases}
\]
\end{enumerate}
\end{lemma}

\begin{proof}
We prove part~(1).

We first prove that, for $1\leq j\leq n-2$,
\begin{equation}\label{eq-mixed-recursion}
u_j=(u_{j+1}^{-1},b_j).
\end{equation}
For $j=n-2$, the first relation in~(B3) gives
\[
u_{n-2}
=(a_{n-2},a_{n-1})
=(a_{n-1}^{-1},b_{n-2})
=(u_{n-1}^{-1},b_{n-2}).
\]

Now let $1\leq j\leq n-3$, and write
$A=(a_j,a_{j+1})$ and $z=u_{j+2}$.
Lemma~\ref{tail-action-Dn}, applied with $m=n$ to
$a_1,\ldots,a_{n-1}$, gives
\[
u_j=(A,z),\qquad
(A,u_{j+1})=1,\qquad
(u_j,a_{j+1})=1.
\]
The relation $(a_j,A)=1$ in~(B1), together with
$(A,u_{j+1})=1$, shows that $A$ commutes with both entries in
$u_j=(a_j,u_{j+1})$. Hence $(A,u_j)=1$. Since
$(A^{-1},z)=(z,A)^{A^{-1}}$, we obtain
\[
(A^{-1},z)
=(z,A)^{A^{-1}}
=(u_j^{-1})^{A^{-1}}
=u_j^{-1}.
\]

The first relation in~(B3) gives
$A=(a_{j+1}^{-1},b_j)$. Since
$(x^{-1},y)=x(x^y)^{-1}$, this becomes
$A=a_{j+1}(a_{j+1}^{b_j})^{-1}$, and therefore
$a_{j+1}^{b_j}=A^{-1}a_{j+1}$.
Moreover,
$z=u_{j+2}\in\langle a_{j+2},\ldots,a_{n-1}\rangle$, and every
generator of this subgroup commutes with $b_j$ by~(B2). Thus
$z^{b_j}=z$.

Using $(xy,z)=(x,z)^y(y,z)$, we now obtain
\[
u_{j+1}^{b_j}
=(a_{j+1}^{b_j},z^{b_j})
=(A^{-1}a_{j+1},z)
=(A^{-1},z)^{a_{j+1}}(a_{j+1},z)
=u_j^{-1}u_{j+1}.
\]
In the last equality, we used $(u_j,a_{j+1})=1$ and
$u_{j+1}=(a_{j+1},z)$. Consequently,
\[
(u_{j+1}^{-1},b_j)
=u_{j+1}(u_{j+1}^{b_j})^{-1}
=u_{j+1}(u_j^{-1}u_{j+1})^{-1}
=u_j.
\]
This proves~\eqref{eq-mixed-recursion}.

We now determine the action of $b_r$. If $j=r+1$, then
rewriting~\eqref{eq-mixed-recursion} gives
$u_{r+1}^{b_r}=u_r^{-1}u_{r+1}$.

If $j\geq r+2$, then
$u_j\in\langle a_j,\ldots,a_{n-1}\rangle$. Every generator of this
subgroup commutes with $b_r$ by~(B2), and hence
$u_j^{b_r}=u_j$.

Suppose that $j=r$. Since $(a_r,b_r)=1$ by~(B2), we have
\[
u_r^{b_r}
=(a_r,u_{r+1})^{b_r}
=(a_r,u_r^{-1}u_{r+1}).
\]
Lemma~\ref{tail-action-Dn} gives $(u_r,a_r)=1$, so
$(a_r,u_r^{-1}u_{r+1})=(a_r,u_{r+1})=u_r$. Thus
$u_r^{b_r}=u_r$.

Suppose next that $r\geq2$ and $j=r-1$. Write
$A=(a_{r-1},a_r)$. Lemma~\ref{tail-action-Dn} gives
\[
u_{r-1}=(A,u_{r+1}),
\qquad
(A,u_r)=1.
\]
The first relation in~(B4), with index $r-1$, gives
$(A,b_r)=1$. Therefore
\[
u_{r-1}^{b_r}
=(A,u_{r+1})^{b_r}
=(A,u_r^{-1}u_{r+1})
=(A,u_{r+1})
=u_{r-1}.
\]

Finally, suppose that $1\leq j\leq r-2$. Then $r\geq3$, and the
preceding case gives $u_{r-1}^{b_r}=u_{r-1}$. We proceed by descending
induction. If $b_r$ fixes $u_{k+1}$ for some
$1\leq k\leq r-2$, then $(a_k,b_r)=1$ by~(B2).
Hence $b_r$ fixes both entries in $u_k=(a_k,u_{k+1})$, and therefore
fixes $u_k$. It follows that, for $1\leq j\leq r-2$,
$u_j^{b_r}=u_j$.

This proves part~(1). The same argument, applied to the
$b$-generators and the words $v_j$, and using the second relations
in~(B3) and~(B4) in place of the first, proves part~(2).
\end{proof}

\begin{lemma}\label{commuting-last-row-Aplus}
Under the hypotheses and notation of
Lemma~\ref{mixed-action-Aplus},
$\langle u_1,\ldots,u_{n-1}\rangle$ and $\langle v_1,\ldots,v_{n-1}\rangle$
centralize each other. Equivalently,
\[
(u_j,v_k)=1
\qquad
(1\leq j,k\leq n-1).
\]
\end{lemma}

\begin{proof}
We first prove that $b_{n-1}$ commutes with every $u_j$. By~(B2),
it commutes with $u_{n-1}=a_{n-1}$, and the first relation
in~(B4), with index $n-2$, gives
\[
(u_{n-2},b_{n-1})
=(a_{n-2},a_{n-1},b_{n-1})
=1.
\]
Starting with $u_{n-2}$, we proceed by descending induction. If
$b_{n-1}$ commutes with $u_{j+1}$ for some
$1\leq j\leq n-3$, then it also commutes with $a_j$ by~(B2).
It therefore commutes with
$u_j=(a_j,u_{j+1})$. Thus $b_{n-1}$ commutes with every $u_j$.

The same argument, using the second relation in~(B4) in place of
the first, shows that $a_{n-1}$ commutes with every $v_k$.
Consequently, if $j=n-1$ or $k=n-1$, then
$(u_j,v_k)=1$.

We prove the remaining cases by strong descending induction on
$j+k$. Let $1\leq j,k\leq n-2$, and assume that the assertion holds
for all pairs $(p,q)$ with $1\leq p,q\leq n-1$ and
$p+q>j+k$.

Suppose first that $k\neq j+1$. Applying part~(2) of
Lemma~\ref{mixed-action-Aplus} with $r=j$ gives
$v_k^{a_j}=v_k$, so $a_j$ commutes with $v_k$. The induction
hypothesis, applied to the pair $(j+1,k)$, gives
$(u_{j+1},v_k)=1$. Hence $v_k$ commutes with both entries in
$u_j=(a_j,u_{j+1})$, and therefore $(u_j,v_k)=1$.

It remains to consider $k=j+1$. Since $k\leq n-2$, we have
$v_k=(b_k,v_{k+1})$. Applying part~(1) of
Lemma~\ref{mixed-action-Aplus} with $r=k$ gives
$u_j^{b_k}=u_j$, because $j=k-1\neq k+1$. Thus $u_j$ commutes with
$b_k$. The induction hypothesis, applied to the pair $(j,k+1)$,
gives $(u_j,v_{k+1})=1$. Therefore $u_j$ commutes with both entries
in $v_k=(b_k,v_{k+1})$, and hence $(u_j,v_k)=1$. This completes the
descending induction.
\end{proof}

\subsection{The presentation theorem for $A_n^+$}

\begin{proof}[Proof of Theorem~\ref{main-adjacent-presentation-Aplus}]
By Lemma~\ref{adjacent-generators},
$a_1,\ldots,a_{n-1},b_1,\ldots,b_{n-1}$ generate $A_n^+$, and by
Lemma~\ref{Aplus-relations-valid}, they satisfy
relations~\textup{(S1)}--\textup{(S4)}. It remains to prove that
these relations are sufficient.

The relations in~\textup{(S1)} involving the $a$-generators are
relations~\textup{(R1)}--\textup{(R3)}. Hence
Theorem~\ref{main-adjacent-presentation} gives a presentation of
$D_n$ on $a_1,\ldots,a_{n-1}$. Conjugation by $\theta$ maps $D_n$
isomorphically onto $D_n^{\ell}$ and sends $a_r$ to $b_r$. Hence
the relations in~\textup{(S1)} involving the $b$-generators give a
presentation of $D_n^{\ell}$ on $b_1,\ldots,b_{n-1}$. We prove that
adjoining the mixed relations~\textup{(S2)}--\textup{(S4)} gives a
presentation of $A_n^+$.

For $n=2$, the presentation in the statement reduces to
$\langle a_1,b_1\mid (a_1,b_1)=1\rangle$. Since
$A_2^+=U_2\times H_2\cong\mathbb Z^2$, the result follows in this
case.

Assume that $n\geq3$ and that the theorem holds for $A_{n-1}^+$.
For $1\leq j\leq n-1$, write $u_j=d_{n,j}$ and $v_j=e_{n,j}$. By
Proposition~\ref{triangular-row-decompositions},
\[
A_n^+=(U_n\times H_n)\rtimes A_{n-1}^+,
\]
where $U_n$ is freely generated by $u_1,\ldots,u_{n-1}$, $H_n$ is
freely generated by $v_1,\ldots,v_{n-1}$, and $(u_j,v_k)=1$ for
$1\leq j,k\leq n-1$. Parts~(1) and~(2) of Lemma~\ref{row-action-Aplus}, with $i=n$, give the conjugation relations between the generators of
$U_n\times H_n$ and the adjacent generators of $A_{n-1}^+$.

The standard presentation of a semidirect product therefore gives a
presentation of $A_n^+$ by adjoining to the presentation of
$A_{n-1}^+$ the generators $u_1,\ldots,u_{n-1}$ and
$v_1,\ldots,v_{n-1}$. The additional defining relations are those
in parts~(1) and~(2) of Lemma~\ref{row-action-Aplus}, with $i=n$,
together with
\[
(u_j,v_k)=1
\qquad (1\leq j,k\leq n-1).
\]

For each $1\leq r\leq n-2$, the two nontrivial relations in
part~(1) of Lemma~\ref{row-action-Aplus}, with $i=n$, are equivalent
to
\[
u_r=(a_r,u_{r+1}),
\qquad
v_r=(b_r,v_{r+1}).
\]
Beginning with $j=n-2$ and ending with $j=1$, we use these relations
successively to eliminate $u_j$ and $v_j$. Rename $u_{n-1}$ as
$a_{n-1}$ and $v_{n-1}$ as $b_{n-1}$. We continue to write
$u_{n-1}=a_{n-1}$ and $v_{n-1}=b_{n-1}$. For
$1\leq j\leq n-2$, let
\[
u_j=(a_j,u_{j+1}),\qquad
v_j=(b_j,v_{j+1})
\]
denote the resulting recursive words in the remaining generators.

The preceding eliminations are Tietze transformations, so the current
presentation still defines $A_n^+$. Consequently, every relation
among the current generators that holds in $A_n^+$ is a consequence
of the current relators and may be added by a Tietze transformation.
By Lemma~\ref{Aplus-relations-valid}, we may therefore add all the
relations in~\textup{(S1)}--\textup{(S4)} that involve $a_{n-1}$ or
$b_{n-1}$. All other relations in~\textup{(S1)}--\textup{(S4)}
already occur in the presentation of $A_{n-1}^+$.

The group defined by the current presentation now satisfies
relations~\textup{(S1)}--\textup{(S4)}. These are precisely
relations~(B1)--(B4) in the hypotheses of
Lemma~\ref{mixed-action-Aplus}. The relations remaining from
part~(1) of Lemma~\ref{row-action-Aplus}, with $i=n$, are
\[
(u_j,a_r)=1,\qquad
(v_j,b_r)=1
\qquad
(1\leq r\leq n-2,\ 1\leq j\leq n-1,\ j\neq r+1).
\]
The relations in~\textup{(S1)} ensure that the hypotheses of
Lemma~\ref{tail-action-Dn} hold separately for the $a$-generators
and the $b$-generators. Conclusion~\textup{(C1)} of that lemma,
applied with $m=n$, shows that these relations follow
from~\textup{(S1)}. The relations in part~(2) of
Lemma~\ref{row-action-Aplus}, with $i=n$, follow from parts~(1)
and~(2) of Lemma~\ref{mixed-action-Aplus}, and
Lemma~\ref{commuting-last-row-Aplus} gives
\[
(u_j,v_k)=1
\qquad (1\leq j,k\leq n-1).
\]

Consequently, all the remaining relations introduced by the
semidirect-product presentation may be deleted by Tietze
transformations. The remaining relations are exactly
relations~\textup{(S1)}--\textup{(S4)}.
\end{proof}

\begin{corollary}\label{minimal-generating-set-Aplus}
For every $n\geq2$,
$(A_n^+)^{\mathrm{ab}}\cong\mathbb Z^{2(n-1)}$. Every generating set
of $A_n^+$ contains at least $2(n-1)$ elements. The adjacent
automorphisms
\[
d_{2,1},d_{3,2},\ldots,d_{n,n-1},
\qquad
e_{2,1},e_{3,2},\ldots,e_{n,n-1}
\]
form a generating set of minimum cardinality.
\end{corollary}

\begin{proof}
When the defining relations in
Theorem~\ref{main-adjacent-presentation-Aplus} are written as
relators, all of them lie in the derived subgroup of the free group on
the adjacent generators. Hence they impose no relations on the
abelianization. Therefore
\[
(A_n^+)^{\mathrm{ab}}\cong\mathbb Z^{2(n-1)}.
\]
It follows that every generating set of $A_n^+$ contains at least
$2(n-1)$ elements, and the adjacent generating set has minimum
cardinality.
\end{proof}

\begin{corollary}\label{all-commutator-presentation-Aplus}
For every $n\geq2$, the group $A_n^+$ admits a presentation on the
$2(n-1)$ adjacent generators
$a_1,\ldots,a_{n-1},b_1,\ldots,b_{n-1}$ in which every defining
relator is a single commutator.
\end{corollary}

\begin{proof}
For arbitrary group elements $x,y,z$, direct calculation gives
\[
\begin{aligned}
(y^x,x^{-1}zy^{-1})
&=(x^{-1}yx)^{-1}
  (x^{-1}zy^{-1})^{-1}
  (x^{-1}yx)
  (x^{-1}zy^{-1})\\
&=(x^{-1}y^{-1}x)
  (yz^{-1}x)
  (x^{-1}yx)
  (x^{-1}zy^{-1})\\
&=x^{-1}y^{-1}xyz^{-1}yzy^{-1}\\
&=(x,y)(y^{-1},z)^{-1}.
\end{aligned}
\]
Consequently, $(x,y)=(y^{-1},z)$ if and only if
$(y^x,x^{-1}zy^{-1})=1$.

For the first relation in~\textup{(S3)}, take
$(x,y,z)=(a_r,a_{r+1},b_r)$; for the second relation
in~\textup{(S3)}, take
$(x,y,z)=(b_r,b_{r+1},a_r)$. Thus, for $1\leq r\leq n-2$, these two
relations may be replaced by
\[
\begin{aligned}
(a_{r+1}^{a_r},a_r^{-1}b_ra_{r+1}^{-1})&=1,\\
(b_{r+1}^{b_r},b_r^{-1}a_rb_{r+1}^{-1})&=1.
\end{aligned}
\]
Every other defining relation in
Theorem~\ref{main-adjacent-presentation-Aplus} is already of the form
$(p,q)=1$ for suitable words $p$ and $q$. Hence every defining
relator can be chosen to be a single commutator.
\end{proof}

\end{document}